\documentclass[10pt]{article}

\usepackage{latexsym,color,amsmath,amsthm,amssymb,amscd,amsfonts}
\usepackage{dsfont}
\usepackage{yfonts}

\newtheorem{theorem}{Theorem}
\newtheorem{corollary}{Corollary}

\newtheorem{lemma}{Lemma}

\newtheorem{proposition}{Proposition}
\newtheorem{definition}{Definition}

\def\R{{\mathbb R}}

\def\vol{{\rm vol}}

\def\1{\mathds{1}}

\def\be{\begin{equation}}
\def\ee{\end{equation}}
\def\Hess{{\nabla^2}}
\def\grad{\nabla}

\begin{document}

\title{Floating functions 
}

\author{Ben Li, Carsten Sch\"utt  and Elisabeth M. Werner
\thanks{Partially supported by  NSF grant DMS-1504701}}

\date{}

 \maketitle

\begin{abstract}

We introduce floating bodies for convex,  not necessarily bounded subsets of $\mathbb{R}^n$. This allows us to define floating 
functions for  convex and log concave functions  and  log concave measures.   We establish the
asymptotic behavior of the integral  difference  of a log concave function  and its floating function.
This  gives rise to a new affine invariant  which bears
striking similarities to the Euclidean affine surface area.

\end{abstract}
\vskip 3mm
\section { Introduction}

Two important closely related notions in affine convex geometry are the floating
body and the affine surface area of a convex body.
The floating body of a convex body is obtained by cutting off caps 
of volume less or equal to a fixed positive constant $\delta$.
Taking the right-derivative of the volume of the floating body gives rise to the affine surface area. 
This was established for all convex bodies in all dimensions by Sch\"utt and Werner
in \cite{SchuettWerner1990}.

The affine surface area was introduced by Blaschke in 1923 \cite{Blaschke:1923}.
Due to its important properties, which make it an effective and powerful tool, it
is omnipresent in geometry.
The affine surface area and its generalizations in the rapidly developing $L_p$ and
Orlicz Brunn--Minkowski theory are the focus of intensive investigations (see e.g.,\
\cite{Gardner:2014, Lutwak:1996, SchuettWerner2004,  
WernerYe2010}) 
and have proven to be a valuable tool in a variety of settings, 
among them solutions  for the affine Bernstein and  Plateau problems
by Trudinger and Wang  \cite {TW1,TW2, TW4}.
Totally new  connections opened up 
to  e.g., PDEs and ODEs (see e.g.,  the papers  
\cite{BoeroetzkyLutwakYangZhang2013,  HuangLutwakYangZhang,  LuYZh3, LuYZh5}, lattice polytopes \cite{BoeroetzkyLudwig} and to  concentration of 
volume, e.g., \cite{FGP,  Ludwig2003}. 
\par
A first characterization of affine surface area was achieved by Ludwig and Reitzner
\cite{Ludwig:2010} and had a profound impact on valuation theory of convex bodies. 
That started a line of research 
(see e.g.,  \cite{ Haberl:2012, Ludwig:2010c,Parapatits:2013, 
 Schuster2010})
leading up to the  recent characterization of all centro-affine valuations by Haberl
and Parapatits \cite{Haberl:2014}.

There is a natural inequality associated with affine surface area, the affine isoperimetric
inequality, which states that among all convex bodies, with fixed volume, affine surface 
area is maximized for ellipsoids. 
This inequality has sparked interest into affine isoperimetric inequalities with a 
multitude of results and proved to be
the key ingredient in many problems
(see e.g., \ \cite{Haberl:2009a, Lutwak:2000, Lutwak:2002, 
Ye:2015, Zhang:1999}).
 In particular, it was used to show uniqueness of
self-similar solutions of the affine curvature flow and to study
its asymptotic behavior  by Sapiro \& Tannenbaum \cite{ST1} and by  Andrews \cite{Andrews:1999, An}.

There are numerous other applications for affine surface area, such as, 
the approximation theory of convex bodies by polytopes 
\cite{Boeroeczky:2000a, GroteWerner, PaourisWerner2013,  Reitzner:2002,  SchuettWerner2003}, 
affine curvature flows \cite{Andrews:1999, Ivaki:2013a,Ivaki:2015, SA1, SA2},
information theory \cite{ArtsteinKlartagSchuettWerner,  CFGLSW, CaglarWerner2014,  CaglarWerner2015, LYZ6, PaourisWerner2012, Werner:2012} and 
partial differential equations \cite{Lutwak:1995}.
Very recent developments are the introduction of floating bodies in spherical and hyperbolic space \cite{BesauWerner1, BesauWerner2}. This has already led to applications in approximation of convex bodies by polytopes \cite{BLW:2016}.

\vskip 3mm

\par
The study of log-concave functions is  a natural  extension of convexity theory.
One of the most important discoveries in   recent investigations  in this direction is the functional version of the famous Blaschke-Santal\'o inequality  
\cite{ArtKlarMil, ArtMil, KBallthesis, 
FradeliziMeyer2007, Lehec2009b}.
\par
Here we introduce floating bodies for unbounded convex sets in the same way as for bounded sets by cutting off sets of fixed volume $\delta$. 
We apply that to the epigraph of a convex function,  which is an unbounded convex set. This allows us to define floating functions $\psi_\delta$ for convex functions $\psi$, 
and consequently  for  log concave functions $f$  which are of the form $f=e^{-\psi}$, where $\psi$ is convex. Namely we put $f_\delta= e^{-\psi_\delta}$.
\par
Taking the right-derivative of the integral  of the floating function $f_\delta$ (see below for the definition) of a log concave function $f=e^{-\psi}$ gives rise to a new affine invariant for log concave functions.
This is the content of our main result which reads as follows. There, $\nabla^2 \psi$ denotes the Hessian of the convex function $\psi$.
\par
\noindent
{\bf Theorem 1.}\label{}
{\em Let  $\psi: \R^n \rightarrow \R$ be a  convex function such that 
$0<\int _{\mathbb R^{n}} e^{-\psi(x)}   dx  < \infty $. Let $c_{n+1} =  \frac{1}{2} \left(\frac{n+2}{\vol_{n} (B^{n}_2)}\right)^\frac{2}{n+2}$.
Then 
$$  
\lim _{\delta \rightarrow 0} \frac{ \int _{\mathbb R^{n}}(e^{-\psi(x)}  - e^{-\psi_{\delta}(x)} )  \  dx } {\delta^{2/(n+2)}} =  c_{n+1} \int_{\mathbb R^{n}} \left(\det\left(\nabla^2 \psi (x) \right)\right)^\frac{1}{n+2} \  e^{-\psi(x)} dx.
$$}

The comparison with convex bodies  leads us  to call  the right hand side of this theorem the affine surface area of the log concave function $f$,
$$
as(f) = \int_{\mathbb R^{n}} \left(\det\left( \nabla^2 \psi (x)\right)\right)^\frac{1}{n+2} \  e^{-\psi(x)} dx.
$$
This is further justified as the expression shares many properties of the affine surface area for convex bodies. It is invariant under affine transformations with determinant $1$ and it is a valuation (see below). A slightly different definition  of affine surface area for log concave functions  has been suggested in  \cite{CFGLSW}.   Both coincide in many cases. We compare the two definitions in
section 3.

We lay the foundation for further investigations of floating functions and the
affine surface  area of log concave functions. The authors believe that
both notions are of interest in its own right and will in particular be useful
for applications, such as, isoperimetric inequalities.

\section{Floating functions and floating sets}

Throughout the paper we will use the following notation.
We denote by $B^{n}_2(x,r)$ the $n$-dimensional closed Euclidean  ball centered at $x$ with radius $r$. We write in short $B^{n}_2=B^{n}_2(0,1)$ for the  Euclidean unit ball centered at $0$.
A convex body $K$ in $\mathbb{R}^n$ is a convex compact subset of $\mathbb{R}^n$ with non-empty interior.  $\partial K$ denotes the boundary of $K$ and $S^{n-1} = \partial B^n_2$.
Finally,  $c,  c_0, c_1$ denote absolute constants that may change from line to line.

For background on convex bodies we refer to the books \cite{GardnerBook, SchneiderBook} and for  background on convex function to \cite{Rockafellar, RockafellarWets}.
\par
\noindent
We first recall the definition of the floating body \cite{BaranyLarman:1988, SchuettWerner1990} .
\par
\noindent
Let $H$ be a hyperplane. Then there is  $u \in S^{n-1} $ and $a \in \mathbb{R}$ such that $H =\{x \in \mathbb{R}^n: \langle x, u\rangle  = a\}$. 
 $H^+ =\{x \in \mathbb{R}^n: \langle x, u \rangle \geq  a\}$ and $H^-=\{x \in \mathbb{R}^n: \langle x, u \rangle  \leq  a\}$ are the two closed half spaces determined by $H$. The hyperplane passing through the point $x$ and 
being orthogonal to the vector $\xi$ is denoted by $H(x,\xi)$.
\par
Let $K$ be a convex body in $\mathbb{R}^n$ and let $\delta >0$.   
Then the {\em (convex) floating body }
$K_{\delta}$  
\cite{SchuettWerner1990} of $K$ is
the intersection of all halfspaces $H^+$ whose
defining hyperplanes $H$ cut off  sets of volume at most $\delta$
from $K$, 
\begin{equation} \label{schwimm}
K_{\delta}=\bigcap_{\{H:  \vol_n\left(H^-\cap K\right) \leq \delta \}} {H^+}. 
\end{equation}
The floating body exists, i.e.,  is non-empty,  if $\delta$ is small enough and clearly, $K_0=K$ and $K_{\delta}\subseteq K$,  for all $\delta \geq 0$.
\vskip 3mm
Similarly, we now introduce the analogue notion for convex, closed, {\em not necessarily bounded} sets $C$. 
\begin{definition} [Floating Set]\label{floating set}
Let $C$ be a closed  convex subset of $\mathbb{R}^n$ with non-empty interior. 
For $\delta>0$, we define the  floating set $C_{\delta} $ by
\begin{align*}
		C_{\delta}= \bigcap \left\{ H^+ : \vol_n\left(H^-\cap C\right) \leq \delta \right\},
	\end{align*}
	where $H$ denotes a hyperplane.
\end{definition}
\vskip 2mm
\noindent
It is clear that $C_\delta $ is a closed convex subset of $C$.  
While for a convex body $K$, $K_\delta$ is a proper subset of $K$ if $\delta >0$, it is now possible that $C_\delta=C$ for  $\delta >0$, e.g.,  when $C$ is a halfspace.

The next proposition states a  property  of the floating set $C_\delta$ which we will need later. The proof is  similar to the one given for the floating body in \cite{SchuettWerner1994}
and we omit it.
\par
\begin{proposition}\label{properties}
Let   $C$ be a closed convex subset of $\mathbb{R}^n$.
 For all \(\delta\) such that \( C_\delta\neq \emptyset \) and all \( x_\delta\in \partial(C_\delta)\cap int (C)\) there exists a support hyperplane \(H\) at \( x_\delta\) to \( C_\delta\) such that \(\delta=\vol_n(C\cap H^-)\).\\
\end{proposition}

\subsection{ Log concave  functions}

Let $\psi: \R^n \rightarrow \R $
 be a  convex function.  
We always consider in this paper convex  functions 
$\psi$ such that  $0 <  \int _{\mathbb R^{n}} e^{-\psi(x)} dx < \infty$. 
\par
\noindent
In the general case, when $\psi$ is neither  smooth nor strictly convex, 
the gradient of $\psi$, denoted  by $\nabla \psi$, exists almost everywhere by Rademacher's theorem (see, e.g., \cite{Rademacher}),   and a theorem of Alexandrov \cite{Alexandroff} and Busemann and Feller \cite{Buse-Feller} guarantees the existence of the (generalized) Hessian, denoted  by
$\nabla^2 \psi$, almost everywhere in $\mathbb R^{n}$. The Hessian  is a quadratic form on $\mathbb{R}^n$, and if $\psi$
is a convex function, for almost every $x \in \mathbb{R}^n$ one has, when $y \rightarrow 0$, that 
$$
\psi( x + y) = \psi (x) + \langle \nabla   \psi(x),  y  \rangle + \frac{1}{2}  \langle \nabla^2 \psi(x) (y), y \rangle + o( \|y\|^2).
$$
\vskip 2mm
\noindent
 A function 
$f: \mathbb{R}^n \rightarrow \mathbb{R}_+$ is   log concave, if it is of the form $f= \exp(-\psi )$ where  $\psi: \R^n \rightarrow\R $ is convex. 
Recall also  that a measure $\nu$ with density $e^{-\psi}$ with respect to the Lebesgue measure
is called log-concave if 
$\psi \colon \R^n \to \R$ is a convex function. 
\vskip 2mm
Let $\psi: \R^n \rightarrow \R$
 be a  convex function and let  
$$
\operatorname{epi}(\psi) = \{ (x,y) \in \mathbb{R}^n \times \mathbb{R}:   y \geq \psi(x)\}
$$
be the  epigraph of $\psi$.
Then $\operatorname{epi}(\psi)$  is a  closed  convex set in $\mathbb{R}^{n+1}$ and  for sufficiently small 
$\delta$ its floating sets $\operatorname{epi}(\psi)_\delta$ are, by above, 
\begin{equation} \label{floatingset}
\operatorname{epi}(\psi)_\delta=\bigcap_{\{H:  \vol_{n+1} \left(H^-\cap \operatorname{epi}(\psi)\right) \leq \delta \}} {H^+}. 
\end{equation}
\noindent
It is easy to see that there exists a unique convex function $\psi_\delta: \mathbb{R}^n \rightarrow \mathbb{R}$ 
such that $(\operatorname{epi}(\psi))_\delta =\operatorname{epi}(\psi_\delta)$.  
This leads to the definitions of a floating function  for convex and  log concave  functions.
\vskip 3mm
\begin{definition} 
Let $\psi: \R^n \rightarrow \R$
 be a  convex function.  Let $\operatorname{epi}(\psi)$ be its epigraph . Let $\delta >0$.
\par
\noindent
 (i) The  floating function  of $\psi$ is defined to be this function $\psi_\delta$ such that
\begin{equation}\label{flotfunct}
(\operatorname{epi}(\psi))_\delta = \operatorname{epi}\left(\psi_\delta\right).
\end{equation} 
\par
\noindent
(ii) Let $f(x)= \exp(-\psi (x) )$ be a log concave function. The  floating function $f_\delta$ of $f$ is defined as
\begin{equation}\label{flotlog}
f_\delta (x) = \exp\left(-\psi_\delta (x) \right).
\end{equation}
\end{definition}
\par
\noindent
Note that when $\psi$ is affine,   $\psi_\delta=\psi$ and, for $f=e^{-\psi}$,  $f_\delta=f$.

\vskip 4mm
\section {Main Theorem and consequences}

Let $C$ be a closed convex set in $\R^n$  and let $z \in \partial C$ be such that $N_{C} (z)$, the outer normal vector,  is unique. 
The following  notion  was  introduced for convex bodies in \cite{SchuettWerner1990}. We define it in the same way  for closed convex sets:  We put $r_C(z)$ to be the radius of the biggest 
Euclidean ball contained in $C$ that touches $C$ in $z$, 
\begin{equation} \label{rC}
r_C(z) = \max\{\rho: B^{n}_2 (z - \rho N_{C} (z), \rho ) \subset C\}.
\end{equation}
$r_{C}$ is called the rolling function of $C$.
If $N_C(z)$ is not unique, $r_C(z) =0$. If $C=\operatorname{epi}(\psi)$, we  will  use, from now on,  the notation  
\begin{equation} \label{r(x)}
r_\psi(x) = r _{\operatorname{epi}(\psi)}((x, \psi(x)).
\end{equation}
Since $\psi$ is continuous, the epigraph of $\psi$ is a closed set.
For functions $\psi$ such that $e^{-\psi}$ is integrable,  we have
that $r_{\psi}(z)$ is bounded  and for almost every $x \in \mathbb{R}^n$, $(x, \psi(x)$ is an element of a  Euclidean ball contained in the epigraph of $\psi$. 
\par
For the remainder of the paper,  $c_{n+1} $ will always be
\begin{equation}\label{ConstTh1-1}
c_{n+1} =  \frac{1}{2} \left(\frac{n+2}{\vol_n (B^n_2)}\right)^\frac{2}{n+2}. 
\end{equation}
\vskip 2mm
\begin{theorem}\label{theo:f-deltafloat}
Let  $\psi: \R^n \rightarrow \R $ be a convex function such that 
$0<\int _{\mathbb R^{n}} e^{-\psi(x)}   dx  < \infty $. 
Then with $c_{n+1}$ as in (\ref{ConstTh1-1}),
\begin{equation}\label{T1}  
\lim _{\delta \rightarrow 0} \frac{ \int _{\mathbb R^{n}}(e^{-\psi(x)}  - e^{-\psi_{\delta}(x)} )  \  dx } {\delta^{2/(n+2)}} =  c_{n+1} \int_{\mathbb R^{n}} \left(\det\left( \nabla^2 \psi (x) \right)\right)^\frac{1}{n+2} \  e^{-\psi(x)} dx.
\end{equation}
\end{theorem}
\vskip 2mm
\noindent
We note that  under the assumptions of the theorem,  the expression on the right hand side of the theorem is finite. 
Indeed, 
for a convex  function $\psi$ the following formulas  hold for the Gaussian curvature $\kappa_{\psi} (z)$  and the outer unit normal $N_{\psi} (z)$ in $z=(x, \psi(x)) \in \partial \operatorname{epi}(\psi)$  (see, e.g., \cite{CFGLSW}), 
\begin{equation} \label{curvature}
\kappa_{\psi} (z) = \frac{\det (\nabla^2 \psi(x))}{\left(1 + \|\nabla \psi (x)\|^2\right)^\frac{n+2}{2}}.
\end{equation}
and
\be\label{normal}
\langle N_{\psi} (z), e_{n+1}\rangle=\frac{1}{(1+\|\nabla \psi (x)\|^2)^{\frac{1}{2}}}.
\ee
As $ \kappa_{\psi} (z) = \prod_{i=1}^n \frac{1}{\rho^\psi_i(z)}$, where $\rho^\psi_i(z)$, $1 \leq i \leq n$,  are the principal radii of curvature, we have for almost all $x \in\Omega_\psi$ that  $r_\psi(x) \leq \frac{1}{\left(\kappa_{\psi} (z)\right)^\frac{1}{n}}$ . 
With (\ref{curvature}) we thus get 
$$
r_\psi(x) \leq \frac{1}{\left(\kappa_{\psi} (z)\right)^\frac{1}{n}}  = \frac {\left(1 + \|\nabla \psi (x)\|^2\right)^\frac{n+2}{2n}} {\left(\det \nabla^2 \psi(x)\right)^\frac{1}{n}}.
$$
Therefore
\begin{eqnarray*}
\int _{\mathbb R^{n}}\left(\det\left( \nabla^2( \psi (x))\right) \right)^\frac{1}{n+2} \  e^{-\psi (x) }\   dx 
&\leq&  \int_{\mathbb R^{n}} \frac{ \left(1 + \|\nabla \psi (x)\|^2\right)^\frac{1}{2}}{r_\psi(x) ^\frac{n}{n+2}} \ e^{-\psi (x) }\    dx 
\end{eqnarray*}
and  we prove in Lemma \ref{integral}  that the last integral is finite. 
\par
\noindent
If the determinant of the Hessian of $\psi$ is $0$ almost everywhere, the right hand term of the theorem is $0$.  This is in particular the case when $\psi$ is piecewise affine. As noted above, the left hand side of the theorem and the proposition will then also be $0$.
\vskip 3mm
\noindent
We postpone the proof of the theorem and  discuss  some consequences first.  The next Proposition will follow  from the lemmas needed for the proof of Theorem \ref{theo:f-deltafloat}.
\par
\noindent
\vskip 2mm
\begin{proposition}\label{corollary}
Let  $\psi: \R^n \rightarrow \R $ be a  convex function such that 
$0<\int _{\mathbb R^{n}} e^{-\psi(x)}   dx  < \infty $. 
Then with $c_{n+1}$ as given by (\ref{ConstTh1-1}),
\begin{equation*}
 \lim _{\delta \rightarrow 0} \frac{ \int_{\mathbb R^{n}} \left| \psi_\delta(x) - \psi (x)\right| \   e^{-\psi} \  dx} {\delta^{2/(n+2)}} =  c_{n+1} \int_{\mathbb R^{n}} \left(\det\left( \nabla^2 \psi (x) \right)\right)^\frac{1}{n+2} \  e^{-\psi(x)} dx.
\end{equation*}
\end{proposition}

\vskip 3mm
\noindent
We call the quantity   $\int _{\mathbb R^{n}}\left(\det\left( \nabla^2 \psi (x) \right)\right)^\frac{1}{n+2} \  e^{-\psi(x)}   \   d (x)$
the {\em affine surface area} of  the log concave function $f=e^{-\psi}$ or, equivalently, 
the affine surface area of  the log concave measure $ d\mu=f dm$, 
\begin{equation}\label{asf}
\operatorname{as}(f) = \operatorname{as} (\mu) =\int_{\mathbb R^{n}} \left(\det\left( \nabla^2 \psi (x) \right)\right)^\frac{1}{n+2} \  e^{-\psi(x)}   \   d (x).
\end{equation}
In view of Corollary \ref{corollary}, another point of view is to interpret  the expression on the right hand side  as the  limit of the weighted (with weight $e^{-\psi}$) ``volume" difference  of the convex  function $\psi$  and its floating function $\psi_\delta$ and call the expression on the right hand side  of the theorem and the corollary the affine surface area  $as(\psi)$ of the convex function of $\psi$. 
\par
\noindent
We now give  reasons why we call  (\ref{asf})  affine surface area. We first recall the definition  of  the $L_p$-affine surface areas $as_p(K)$ for convex bodies $K$.
For $- \infty \leq p \leq \infty$, $p \neq -n$, they are  defined as \cite{Blaschke:1923,   Lutwak:1996, SchuettWerner2004}
\begin{equation}\label{asp-K}
\operatorname{as}_p(K) = \int_{\partial K} \frac{ \kappa_K(z)^\frac{p}{n+p}}{\langle z, N_K(z) \rangle^\frac{n(p-1)}{n+p}} d\mu_K(z).
\end{equation}
Here, $N_K(z)$ is the outer unit normal at $z \in \partial K$, $\mu_K$ is the usual surface area measure on $\partial K$  and $\kappa_K(z)$ is the Gauss curvature at $x$.
In particular, for $p=1$ we get the (usual) affine surface area of $K$,
\begin{equation}\label{asK}
\operatorname{as}(K) =  \int _{\partial K } \left( \kappa_{K}(z)\right)^\frac{1}{n+1}  d_{\mu_K}(z).
\end{equation}
\noindent
We pass from integration over $\mathbb{R}^n$ in (\ref{asf}) to integration over $\partial \operatorname{epi}(\psi)$
with the change of variable formula $\left(1 + \|\nabla \psi (x)\|^2\right)^\frac{1}{2}dx = d \mu_{\operatorname{epi}(\psi)}$. With (\ref{curvature}) we get
\begin{equation}
 \operatorname{as} (f) =\int_{\partial \operatorname{epi}(\psi)} \left(\kappa_{\partial \operatorname{epi}(\psi)} (z)\right)^\frac{1}{n+2} 
e^{-\langle z,e_{n+1}\rangle}d \mu_{\operatorname{epi}(\psi)}(z).
\end{equation}
Thus the expression  (\ref{asf})  coincides (for the  unbounded convex set $\operatorname{epi}(\psi)$)  with the one  for the affine surface area of a convex body in $\mathbb{R}^{n+1}$,
given in (\ref{asK}).
This is one reason  to call the quantity the affine surface area of $f$.  
\par
\noindent
Moreover, $as(f)$  has similar properties as  $as_1(K)$. Firstly,  an  affine invariance property holds
(with the same degree of homogeneity as the affine surface area for convex bodies in $\mathbb{R}^{n+1}$): For all   affine transformations $A: \R^n \rightarrow \R^n$ such that $\det A \neq 0$
$$
\operatorname{as}(f \circ A) = |\det A |^{-\frac{n}{n+2}} \   \operatorname{as}(f) .
$$
This identity is  easily checked using   $\Hess_x (\psi\circ A)=A^t\Hess_{Ax}\psi A$.
\newline
Secondly,  as for convex bodies, a valuation property holds for $as(f)$, i.e., we have for log concave functions $f_1=e^{-\psi_1}$ and  $f_2=e^{-\psi_2}$ 
that 
\[ \operatorname{as}(f_1)+ \operatorname{as}(f_2) =  \operatorname{as}(\max (f_1, f_2)) + as(\min (f_1, f_2)), 
\]
provided  $\min
(\psi_1,\psi_2)$ is convex. 
\par
\noindent
Another reason comes from  the next  observation which shows that the definition for affine surface area for a function agrees with the definition for convex bodies if the function is the 
gauge function $\| \cdot \|_K$ of a  convex body $K$ with $0$ in its interior, 
$$
\|x\|_K = \min\{ \alpha\geq 0: \ x \in  \alpha K\} = \max_{y \in K^\circ} \langle x, y \rangle = h_{K^\circ} (x). 
$$ 
If $\psi(x) = \frac{\|x\|_K^2}{2}$, then 
\begin{equation}\label{as-Gleichung}
 \   \operatorname{as}\left(\frac{\|\cdot\|_K^2}{2}\right) = \frac{(2\pi)^\frac{n}{2}}{\vol_n (B_2^n)} \  \operatorname{as}_{\frac{n}{n+1}}(K).
\end{equation}
This was already observed in \cite{CFGLSW}.  There,  a slightly different definition of affine surface area for log concave functions $f=e^{-\psi}$ was introduced, namely
\begin{equation}\label{AffIsoIq}
  \int_{\mathbb R^{n}}e^{-(\frac{n}{n+2}\psi(x)+\frac{1}{n+2}\langle x, \nabla\psi(x)\rangle)}
\left(\det \, \Hess \psi (x)\right)^{\frac{1}{n+2}} dx.
\end{equation}
  Both definitions coincide for $2$-homogeneous functions $\psi$. From
an affine isoperimetric inequality proved in \cite{CFGLSW}
for the expression (\ref{AffIsoIq}) we can deduce the following corollary.

\begin{corollary}
Let  $\psi: \R^n \rightarrow \R $ be a $2$-homogeneous, convex function such that 
$0<\int _{\mathbb R^{n}} e^{-\psi(x)}   dx  < \infty $. Then
$$
\operatorname{as}(f)
\leq(2\pi)^{\frac{n}{n+2}}\left(\int_{\mathbb R^{n}}e^{-\psi(x)}dx\right)^{\frac{n}{n+2}},
$$
with equality if and only if there are $a\in\mathbb R$ and a positive 
definite matrix $A$ such that for all $x\in\mathbb R^{n}$
$$
\psi(x)=\langle Ax,x\rangle +a.
$$
\end{corollary}

We conjecture that this inequality holds for general, convex functions.

We include the argument for (\ref{as-Gleichung})  for completeness.
\newline
We integrate in polar coordinates with respect to the normalized cone measure  $\sigma_K$  of  $K$. 
Thus, if we write $x=r\theta$, with $\theta\in\partial K$, then $dx=n \ \vol_n (K) r^{n-1}drd\sigma_K(\theta)$ and we get
\begin{eqnarray*}
 \operatorname{as} \left(\frac{\|\cdot\|_K^2}{2}\right)&=&  \int_{\mathbb R^{n}} \det\left( \nabla^2( \psi (x)) \right)^\frac{1}{n+2} \  e^{-\psi(x)} dx\\
 &=&n \ \vol_n (K) \int_0^{+\infty}r^{n-1}e^\frac{-r^2}{2}dr \int_{\partial K}  \left(\det \,\nabla^2  \psi (\theta) \right)^\frac{1}{n+2} \ d\sigma_K(\theta)\\
 &=& (2\pi)^\frac{n}{2}\frac{\vol_n (K)}{\vol_n (B_2^n) }\int_{\partial K}   \left(\det \, \nabla^2 \psi (\theta) \right)^\frac{1}{n+2} \ d\sigma_K(\theta).
\end{eqnarray*}
The relation between the normalized cone measure $\sigma_K$ and the Hausdorff measure $\mu_K$ on $\partial K$ is given by 
$$
d\sigma_K(x)=\frac{\langle \theta,N_K(\theta)\rangle d\mu_K(\theta)}{n \   \vol_n (K)}.
$$
 E.g., Lemma 1  of \cite{CFGLSW} (and its proof) show that 
$
\det \, \nabla^2 \psi (\theta) = \frac{\kappa_K(\theta)}{\langle \theta,N_K(\theta)\rangle ^{n+1}}
$.
Thus 
\begin{eqnarray*}
\operatorname{as}\left(\frac{\|\cdot\|_K^2}{2}\right)&=& \frac{(2\pi)^\frac{n}{2}}{n \  \vol_n (B_2^n)}\int_{\partial K}  \left(\frac{ \kappa_K (x) }{ \langle x,N_K(x)\rangle^{n+1} }\right)^\frac{1}{n+2} \langle x,N_K(x)\rangle d\mu_K(x)\\
&=& \frac{(2\pi)^\frac{n}{2}}{n \  \vol_n (B_2^n)}\  \operatorname{as}_\frac{n}{n+1}(K).
\end{eqnarray*}
\par
\noindent
Finally, the most compelling  reason to call the quantity  $as(f)$ affine surface area
is the following theorem proved in \cite{SchuettWerner1990} in the case of convex bodies in $\mathbb{R}^{n}$.
$$
\lim_{\delta \rightarrow 0} \frac{\vol_n(K) - \vol_n(K_\delta)}{\delta^\frac{2}{n+1}} = c_n \  \int _{\partial K } \left( \kappa_{K}(z)\right)^\frac{1}{n+1}  d_{\mu_K}(z),
$$
where $c_n=\frac{1}{2} \left(\frac{n+1}{\vol_{n-1} (B^{n-1}_2)}\right)^\frac{2}{n+1}$. 
Theorem \ref{theo:f-deltafloat} is  its analogue for  log concave functions.
Thus this theorem provides  a  geometric description   of  affine  surface area for such functions.

\section {Proof of Theorem
\ref{theo:f-deltafloat} \label{proofs1}}

We need several lemmas. 
The first lemma is standard and well known (see, e.g. \cite{SchuettWerner2003}).  
\vskip 2mm 
 \begin{lemma}\label{lemma:cap}
 Let
$$
\mathcal{E}=\left\{x\in\mathbb{R}^n:
\ \sum_{i=1}^{n}\left|\frac{x_{i}}{a_{i}}\right|^{2}\leq
1 \right\}.
$$ 
and let $H_h=H((a_{n}-h)e_{n},e_{n})$. Then for all
$h\leq a_{n}$ 
\begin{eqnarray*}
h^\frac{n+1}{2} \ \left( 1- \frac{h}{2a_n}\right)^\frac{n-1}{2} \  
\leq \  
\frac{ (n+1) \  a_n^\frac{n-1}{2}\   \vol_{n}(\mathcal{E}\cap H_h^{-})}{2^\frac{n+1}{2}
\vol_{n-1}(B_{2}^{n-1})\  \prod_{i=1}^{n-1}a_{i}}\  
\leq\  
h^\frac{n+1}{2}.
\end{eqnarray*}
In particular, if $\mathcal{E}= r B^n_2$ is a Euclidean ball  with radius  $r$  in $\mathbb{R}^n$, then for for all $u \in S^{n-1}$, for $h \leq r$ and $H_h=H((r-h)u, u )$, 
\begin{eqnarray*}
 h^\frac{n+1}{2}  \left( 1-\frac{ h}{2r} \right)  \  \leq \   \frac{(n+1)\   \vol_{n}\left(rB^n_2\cap H_h^{-}\right)}{2^\frac{n+1}{2} \ \vol_{n-1}\left(B^{n-1}_2\right) \  r^\frac{n-1}{2}}
  \leq 
  \ h^\frac{n+1}{2}. 
\end{eqnarray*}
 \end{lemma}
 \vskip 4mm
 \noindent
The next  lemma is  well known. We refer to e.g., \cite{RockafellarWets}.
\vskip 2mm
  \begin{lemma}\label{integrate+limit}
Let  $\psi: \R^n \rightarrow \R $ be a  convex function. Then
$\int_{}   e^{-\psi(x)}   dx  < \infty $ if and only if  for some  $\gamma >0$ there exists $\beta \in (-\infty, \infty)$ such that 
for all $x \in \mathbb{R}^n$, 
\begin{equation}\label{levelcoercive}
\psi(x) \geq \gamma \|x\| + \beta.
\end{equation}
\end{lemma}
\par
\noindent
In particular it follows from (\ref{levelcoercive})  that
\begin{equation}\label{coercive}
\lim_{\|x\| \rightarrow \infty} \psi(x) = \infty.
\end{equation}
This property is sometimes called coercivity (see, e.g., \cite{RockafellarWets}).

\vskip 4mm
 \begin{lemma}\label{lemma:psi->grad}
Let  $\psi: \R^n \rightarrow \R $ be a  convex function.
If  $0 < \int_{\mathbb R^{n}}   e^{-\psi(x)}   dx < \infty$, then $\int_{\mathbb R^{n}}   e^{-\psi(x)}  \left( 1 +\| \grad \psi(x) \|^2\right) ^\frac{1}{2}   dx < \infty$. 
\end{lemma}
\vskip 3mm
\begin{proof}
\begin{eqnarray*}
&&\int_{\mathbb R^{n}}   e^{-\psi(x)}  \left( 1 +\| \grad \psi(x) \|^2\right) ^\frac{1}{2}   dx  \leq \int_{\mathbb{R}^n}   e^{-\psi(x)}  \left( 1 +\| \grad \psi(x) \|\right)   dx \\
&&= \int_{\mathbb R^{n}}   e^{-\psi(x)} dx   + \int_{\mathbb R^{n}}   e^{-\psi(x)} \left(\sum_{i=1}^n \left | \frac{\partial  \psi}{\partial x_i}(x) \right |^2\right) ^\frac{1}{2}   dx 
\end{eqnarray*}
The first integral is finite by assumption. We consider the second integral.
\begin{eqnarray*}
&&\int_{\mathbb R^{n}}   e^{-\psi(x)} \left(\sum_{i=1}^n \left | \frac{\partial  \psi}{\partial x_i}(x) \right |^2\right) ^\frac{1}{2}   dx \leq \int_{\mathbb R^{n}}   e^{-\psi(x)} \sum_{i=1}^n \left | \frac{\partial  \psi}{\partial x_i}(x) \right |   dx \\
&&=  \sum_{i=1}^n  \int_{\mathbb R^{n}}   e^{-\psi(x)} \left | \frac{\partial  \psi}{\partial x_i}(x) \right |   dx. 
\end{eqnarray*}
Let $ y = (x_2, \cdots, x_n) \in \mathbb{R}^{n-1}$ and let $m(y) \in \mathbb{R}$ satisfy 
$$
\psi((m(y), y) = \min _{x_1 \in \mathbb{R} }\psi(x_1, y).
$$
By Lemma \ref{integrate+limit}, $\psi$ satisfies (\ref{coercive}).  This means that $\psi$ has a global minimum (which needs not be unique,  unless $\psi$ is strictly convex)
and therefore, $m(y)$ exists. Then
\begin{eqnarray*}
&& \int_{\mathbb R^{n}}   e^{-\psi(x)} \left | \frac{\partial  \psi}{\partial x_1}(x) \right |   dx_1 \cdots dx_n  \\
&& = \int_{\mathbb R^{n-1}}  \left( \int_{m(y)} ^\infty  e^{-\psi(x_1, y)}  \frac{\partial  \psi}{\partial x_1}(x_1, y)    dx_1 \right)  dy - \int_{\mathbb R^{n-1}}  \left( \int_{-\infty}^{m(y)}  e^{-\psi(x_1, y)} \frac{\partial  \psi}{\partial x_1}(x_1, y)    dx_1 \right)  dy \\
&& = \int_{\mathbb R^{n-1}}  \left( \int_{m(y)} ^\infty  -\frac{\partial}{\partial x_1} \left( e^{-\psi(x_1, y)} \right)   dx_1 \right)  dy +  \int_{\mathbb R^{n-1}}  \left( \int_{-\infty} ^{m(y)}   \frac{\partial}{\partial x_1} \left( e^{-\psi(x_1, y)} \right)   dx_1 \right)  dy \\
&& = 2 \   \int_{\mathbb R^{n-1}} e^{-\psi(m(y), y)}   dy.
\end{eqnarray*}
By Lemma  \ref{integrate+limit}, one has for all $x_1 \in \mathbb{R}$ that $\psi(x_1, y) \geq \gamma \|(x_1,y)\| + \beta$ for some $\gamma >0$ and some $\beta$.  It follows that $\psi(m(y), y) \geq \gamma \|(m(y),y) \| + \beta \geq  \gamma \|y \| + \beta$. It follows that this term is finite.
\par
\noindent
The other coordinates are treated similarly. This finishes the proof of the lemma.
\end{proof} 
\vskip 4mm
 \begin{lemma}\label{lemma:f-delta}
(i)  Let $x \in \mathbb R^{n}$ be such that the Hessian $\nabla^2 \psi$ at $x$  is positive definite. Then there are constants $\beta_1$ and $\beta_2$ such that  for all $\varepsilon >0$ there is $\delta_0= \delta_0(x, \varepsilon)$ such that for all $\delta \leq \delta_0$, 
\begin{eqnarray*}
(1- \beta_2 \varepsilon) \ 
c_{n+1}  \left(\det\left( \nabla^2 \psi (x)\right) \right)^\frac{1}{n+2} \leq \frac{\psi_\delta (x) - \psi(x) }{\delta^\frac{2}{n+2} }  \leq (1+ \beta_1 \varepsilon) \ 
c_{n+1}  \left(\det\left( \nabla^2 \psi (x)\right) \right)^\frac{1}{n+2},
\end{eqnarray*}
where  $c_{n+1}$ is given by (\ref{ConstTh1-1}).  Consequently, for $f = e^{-\psi}$ we get with (new)  constants $\beta_1$ and $\beta_2$
\begin{eqnarray*}
(1- \beta_2 \varepsilon)   f(x) \  c_{n+1}  \left(\det\left( \nabla^2\psi (x)\right) \right)^\frac{1}{n+2} &\leq&  \frac{ f(x) - f_\delta (x) }{\delta^\frac{2}{n+2} } \\ 
&\leq & (1+ \beta_1 \varepsilon)   f(x) \  c_{n+1}  \left(\det\left( \nabla^2\psi(x) \right) \right)^\frac{1}{n+2} .
\end{eqnarray*}
\vskip 2mm
\noindent
(ii)  Let $x \in \mathbb R^{n}$ be such that $ \det\left(\nabla^2 \psi(x)\right) =0$. Then for all $\varepsilon >0$ there is $\delta_0= \delta_0(x, \varepsilon)$ such that for all $\delta \leq \delta_0$, 
\begin{eqnarray*}
0 \leq \frac{\psi_\delta (x) - \psi(x) }{\delta^\frac{2}{n+2} }  \leq \varepsilon.
\end{eqnarray*}
Consequently,  for $f = e^{-\psi}$ we get for all $\varepsilon >0$ that there is $\delta_0= \delta_0(x, \varepsilon)$ such that for all $\delta \leq \delta_0$, 
\begin{eqnarray*}
0 \leq \frac{f(x) - f_\delta (x) }{\delta^\frac{2}{n+2} }  \leq \varepsilon.
\end{eqnarray*}
 \end{lemma}
\vskip 3mm
\begin{proof}
Let $0 < \varepsilon <\frac{1}{n^2}$ be given and let 
$x_0 \in \mathbb R^{n}$.  We put $z_{x_0}=(x_0, \psi(x_0))$. 
Denote by $N_\psi(z_{x_0})$ the outer unit normal in $z_{x_0}$  to the surface described by $\psi$.  As recalled above, $N_\psi(z_{x_0})$ exists uniquely for almost all $x_0$.
\par
(i) We assume that $x_0$ is such that the Hessian $\nabla^2 \psi (x_0)$  is positive definite. Then, 
locally around $z_{x_0}$, the graph of $\psi$ can be approximated by an ellipsoid $\mathcal{E}$. We  make this precise:
\newline
Let $\mathcal{E}$ be such that the lengths of its principal axes are $a_1, \dots, a_{n+1}$ and such that its center 
is at $z_{x_0} - a_{n +1} N_\psi(z_{x_0})$.  Let $\mathcal{E} (\varepsilon^-)$ be the ellipsoid centered at $z_{x_0}-  a_{n+1} N_\psi(z_{x_0})$ 
whose principal axes  coincide with the ones of $\mathcal{E}$,  but have lengths $(1-\varepsilon) a_1, \dots, (1-\varepsilon) a_{n}, a_{n+1}$. Similarly, let $\mathcal{E} (\varepsilon^+)$ be the ellipsoid centered at $z_{x_0}-  a_{n+1} N_\psi(z_{x_0})$, with the same principal axes as $\mathcal{E}$,  but with lengths $(1+\varepsilon) a_1, \dots, (1+\varepsilon) a_{n}, a_{n+1}$.
Then
$$ z_{x_0} \in \partial \mathcal{E}  \hskip 3mm \text { and } \hskip 3mm N_{\mathcal{E}}(z_{x_0}) = N_\psi(z_{x_0}),$$
and (see, e.g., \cite{SchuettWerner2003})  there exists a $\Delta_\varepsilon >0$ such that the hyperplane $H\left( z_{x_0} - \Delta_\varepsilon N_\psi(z_{x_0}), N_\psi(z_{x_0})\right)$ such that 
\begin{eqnarray}\label{ellipse}
&& \hskip -10mm H^-\left( z_{x_0} - \Delta_\varepsilon N_\psi(z_{x_0}), N_\psi(z_{x_0})\right)  \  \cap  \  \mathcal{E} (\varepsilon^-) \nonumber \\
&&  \subseteq  H^-\left( z_{x_0} - \Delta_\varepsilon N_\psi(z_{x_0}), N_\psi(z_{x_0})\right) \  \cap \   \{(x,y):   y \geq \psi(x)\} \nonumber \\
&& \hskip 10mm \subseteq H^-\left( z_{x_0} - \Delta_\varepsilon N_\psi(z_{x_0}), N_\psi(z_{x_0})\right) \  \cap  \  \mathcal{E} (\varepsilon^+) .
\end{eqnarray}
For  $\delta \geq 0$, let $z_\delta= (x_0, \psi_\delta(x_0))$. We choose $\delta$ so small that  for all support hyperplanes $H(z_\delta)$ to $\left(\operatorname{epi}(\psi)\right)_\delta$  through $z_\delta$ we have
$$
  H(z_\delta) ^{-}  \cap  \mathcal{E} (\varepsilon^-) \subseteq   H^-\left( z_{x_0} - \Delta_\varepsilon N_\psi(z_{x_0}), N_\psi(z_{x_0})\right)  \  \cap  \  \mathcal{E} (\varepsilon^-) .
$$
Let $\Delta_{\delta}$ be such that
\begin{equation}\label{DefDeltadelta}
H(z_{x_{0}}-\Delta_{\delta}N_{\psi}(z_{x_{0}}),N_{\psi}(z_{x_{0}}))
\end{equation}
is a supporting hyperplane to $\operatorname{epi}(\psi_{\delta})$.
Moreover, we choose $\delta$ so small that $\Delta_\delta \leq \Delta_\varepsilon$ of  (\ref{ellipse}).
 As $\partial \operatorname{epi}(\psi)$ is approximated by an ellipsoid in $z_{x_0}$, we have that $z_\delta \in \text{int}(\operatorname{epi}(\psi))$.  
 Thus we get 
by definition of $\left(\operatorname{epi}(\psi)\right)_\delta$, respectively $\psi_\delta$, by Proposition \ref{properties} and Lemma \ref{lemma:cap},
\begin{eqnarray*}
\delta &\leq&  \vol_{n+1} \left(H\left( z_{x_0} - \Delta_\delta N_\psi(z_{x_0}), N_\psi(z_{x_0})\right) \  \cap \ \operatorname{epi}(\psi)   \right) \\
&\leq & \vol_{n+1} \left(H\left( z_{x_0} - \Delta_\delta N_\psi(z_{x_0}), N_\psi(z_{x_0})\right)   \cap \   \mathcal{E} (\varepsilon^+) \right)  \\ 
&\leq& (1+\varepsilon)^{n} \   \frac{2^\frac{n+2}{2} \vol_n(B^n_2)} {n+2}  \   \prod_{i=1}^n \frac{a_i}{\sqrt{a_{n+1}}}  \ \Delta_\delta ^\frac{n+2}{2}.
\end{eqnarray*} 
As
$\kappa_{\psi} (z_{x_0}) =  \prod_{i=1}^n \frac{a_{n+1}}{a_i^2}$ (see, e.g., \cite{SchuettWerner2003}),  (\ref{curvature}) yields
\begin{eqnarray}\label{below1}
\Delta_\delta \geq  \frac{c_{n+1} }{(1+ \varepsilon)^{ \frac{2n}{n+2}}} \   \frac{\left(\det \nabla^2 \psi(x_0) \right)^\frac{1}{n+2}}{\left(1 + \|\nabla \psi(x_0)\|^2\right)^\frac{1}{2} } \   
\delta^\frac{2}{n+2} , 
\end{eqnarray}
where $c_{n+1}$ is as given by (\ref{ConstTh1-1}).  By (\ref{DefDeltadelta}) 
$$
\Delta_\delta \leq  \langle N_{\psi} (z), e_{n+1}\rangle \left( \psi_\delta(x_0)  - \psi(x_0)\right) .
$$
Therefore, with (\ref{normal}),
$$
\Delta_\delta \leq  \frac{\psi_\delta(x_0)  - \psi(x_0)}{(1+\|\nabla \psi(x_0)\|^2)^{\frac{1}{2}}} 
$$
and thus with (\ref{below1}) that 
\begin{equation}\label{below2} 
\psi_\delta(x_0)  - \psi(x_0)  \geq \frac{c_{n+1} }{(1+ \varepsilon)^{ \frac{2n}{n+2}}} \   \left(\det \nabla^2 \psi (x_0)\right)^\frac{1}{n+2}\  \delta ^\frac{2}{n+2}.
\end{equation}
\par
Now we estimate $\delta$ from below. By Proposition \ref{properties},  there exists a hyperplane $H_\delta$ such that $\delta = \vol_{n+1}(H_\delta^- \cap G_\psi)$. By  (\ref{ellipse}),
\begin{eqnarray} \label{unten}
\delta &\geq&  
\vol_{n+1}\left( H_\delta^- \  \cap \   \mathcal{E} (\varepsilon^-) \right).
\end{eqnarray}
The expression $\vol_{n+1}\left( H_\delta^- \  \cap \   \mathcal{E} (\varepsilon^-) \right)$ is invariant under affine transformations with determinant $1$. We apply an affine transformation that 
maps $\mathcal {E} (\varepsilon^-)$ into a Euclidean ball with radius 
\begin{equation} \label{radius}
r= (1- \varepsilon) \  \left(\frac{1} {\kappa_{\psi} (z_{x_0})} \right)^ \frac{1}{n}.
\end{equation}
Now we use   Lemma 11 of \cite{SchuettWerner1990}. Please note that $z_{x_0}$ corresponds to $0$ of Lemma 11, that $z_\delta$ corresponds to $z$ and that $N_\psi(z_{x_0})$
corresponds to $N(0)=(0, \cdots, 0,-1)$. We choose $\varepsilon <\varepsilon_0$, where  $\varepsilon_0$ is given by Lemma 11,  and we choose $\delta$ so small that 
$\psi_\delta(x_0) - \psi(x_0) = \|z_\delta - z(x_0)\| \leq \varepsilon <\varepsilon_0$.
By  Lemma 11 (iii)  of \cite{SchuettWerner1990},
\begin{eqnarray*}
\vol_{n+1}\left( H_\delta^- \  \cap \   \mathcal{E} (\varepsilon^-) \right) &=& \vol_{n+1}\left( H_\delta^- \  \cap \  B^{n+1}_2\left( z_{x_0}- r\  N_\psi(z_{x_0}),  r\  \right)\right) \\
& \geq& \eta(\gamma)^{-n} \  \vol_{n+1}\left(C( r, d_0 (1- c(\eta(\gamma)-1)))\right),
\end{eqnarray*}
where $c$ is an absolute constant, $C( r, d_0 (1- c(\eta(\gamma)-1)))$ is the cap of the $(n+1)$-dimensional Euclidean ball $B^{n+1}_2\left( z_{x_0}- r\  N_\psi(z_{x_0}),  r\  \right)$ of height $d_0 (1- c(\eta(\gamma)-1)))$
and $d_0$ is the  distance from $z_\delta$ to the boundary of $B^{n+1}_2\left( z_{x_0}- r\  N_\psi(z_{x_0}),  r\  \right)$. $\gamma = 4\sqrt{2rd_0}$ and $\eta$ is  a monotone function on $\mathbb{R}^+ $ such that $\lim_{t \rightarrow 0} \eta(t) =1$.
Thus, by Lemma \ref{lemma:cap}, 
\begin{equation}\label{delta1}
\delta 
\geq    \frac{2^\frac{n+2}{2} \vol_{n}(B^n_2)} {(n+2)\  \eta(\gamma)^n}  r^\frac{n}{2}    \  \left(d_0 (1- c(\eta(\gamma)-1)))\right)^\frac{n+2}{2}  
\left(1- \frac{d_0 (1- c(\eta(\gamma)-1))) }{2r}\right)^\frac{n}{2}.
\end{equation}
We apply Lemma 11 (ii)  of \cite{SchuettWerner1990} next. Please note that $z_n$ of Lemma 11  corresponds to  
$z_n= \langle e_{n+1}, N_\psi(z_{x_0}) \rangle   \left(\psi_\delta(x_0) - \psi(x_0) \right)$ in our case and $\frac{\xi }{\|\xi\| }= e_{n+1}$. 
Then by Lemma 11 (ii), 
\begin{equation}\label{d0}
d_0 \leq  \langle e_{n+1}, N_\psi(z_{x_0}) \rangle \  \left(\psi_\delta(x_0) - \psi(x_0) \right) \leq d_0 + \frac{2 d_0^2}{r  \ \left| \langle e_{n+1}, N_{\psi}(z_{x_0})\rangle\right|^2}.
\end{equation}
Thus  we get  for sufficiently small $\delta$,  with an absolute constant $\beta_1$, that 
\begin{equation}\label{eta1}
\eta(\gamma) = \eta (4\sqrt{2rd_0})  \leq 1+\beta_1 \varepsilon
\end{equation}
and hence 
\begin{equation}\label{eta2}
1 - c (\eta(\gamma) -1) \geq 1 - \beta_2 \varepsilon, 
\end{equation}
with an  absolute constant $\beta_2$.
It follows from (\ref{d0}) that 
$$
d_0 \geq \langle e_{n+1}, N_\psi(z_{x_0}) \rangle   \left(\psi_\delta(x_0) - \psi(x_0) \right) 
\left( 
1-\frac{2\  \left(\psi_\delta(x_0) - \psi(x_0)\right)}{r\  \left( \langle e_{n+1}, N_\psi(z_{x_0}) \rangle \right)} 
\right).
$$
We conclude with (\ref{normal}), (\ref{delta1}),  (\ref{d0}),  (\ref {eta1})  and (\ref {eta2})  that with   $c_{n+1} = \frac{1}{2} \left(\frac{n+2}{\vol_n (B^n_2)}\right)^\frac{2}{n+2}$ and (new) absolute constants $\beta_1, \beta_2$, 
\begin{eqnarray*}
\delta ^\frac{2}{n+2} 
&\geq& \frac{ 1-\beta_2\   \varepsilon}{(1+\beta_1\   \varepsilon)^\frac{2n}{n+2}}  \  \frac{r^\frac{n}{n+2}}{c_{n+1}}\  \frac{\left(\psi_\delta(x_0) - \psi(x_0) \right)}{(1+\|\nabla \psi(x_0)\|^2)^{\frac{1}{2}}} \left(1- 2 \frac{\psi_\delta(x_0) - \psi(x_0) }{ r  (1+\|\nabla \psi (x_0)\|^2)^{\frac{1}{2}}}\right) ^{2\frac{n+1}{n+2}}.
\end{eqnarray*}
For  $\delta$  small enough, 
(\ref{curvature}) and (\ref{radius}) give with (new) absolute constants $\beta_1, \beta_2$,
\begin{eqnarray}\label{below3}
\psi_\delta(x_0)  - \psi(x_0)  \leq \frac{(1+ \beta_1\ \varepsilon)^\frac{2n}{n+2} }{(1- \beta_2\ \varepsilon)^{2\frac{n+1}{n+2}}}  \ c_{n+1} \  \left(\det \nabla^2 \psi(x_0)\right)^\frac{1}{n+2}  \  \delta ^\frac{2}{n+2}.
\end{eqnarray}
\vskip 3mm
\noindent
(ii) Now we assume that $\det \left(\nabla^2 \psi(x_0)\right) =0$.  Suppose first that there is $\delta_0$ such that $z_{\delta_0} \in \partial \operatorname{epi}(\psi)$. Then 
$z_\delta   \in \partial \operatorname{epi}(\psi)$ for all $\delta \leq \delta_0$.  As $z_\delta = (x_0, \psi_\delta (x_0))$ and $z_{x_0} = (x_0, \psi(x_0))$, we thus get  that  $\psi_\delta(x_0)  = \psi(x_0)$ for all $\delta \leq \delta_0$, 
and hence $\frac{\psi_\delta(x_0)  - \psi(x_0)}{\delta^\frac{2}{n+2}}=0$.
\par
Suppose next that for all $\delta >0$, $z_\delta  \in \text{int}(\operatorname{epi}(\psi))$.  As $\det \left(\nabla^2 \psi(x_0)\right) =0$,  the indicatrix of Dupin at $z_{x_0}$ is an elliptic cylinder 
and we may assume that the first $k$ axes have infinite lengths and the others not. Then, 
(see e.g., \cite{SchuettWerner2004}, Lemma 23, proof),  for all $\varepsilon>0 $ there is an ellipsoid $\mathcal{E}$ and $\Delta_\varepsilon >0$ such that for all $\Delta \leq \Delta_\varepsilon$ 
\begin{eqnarray*} 
\mathcal{E}  \cap H^-(z_{x_0} - \Delta N_\psi(z_{x_0}), N_\psi(z_{x_0}))  \subset \operatorname{epi}(\psi) \cap H^-(z_{x_0} - \Delta N_\psi(z_{x_0}), N_\psi(z_{x_0}))
\end{eqnarray*}
and such that the lengths of the $k$ first principal axes of $\mathcal{E}$ are larger than $\frac{1}{\varepsilon}$.  
By Proposition \ref{properties} there is a hyperplane $H_\delta$  such that 
$z_\delta \in H_\delta$ and such that $\delta = \vol_{n+1}(\operatorname{epi}(\psi) \cap H_\delta^-)$. 
We choose $\delta$ so small  that 
$$
\mathcal{E} \cap H_\delta^- \subset \mathcal{E} \cap H^-(z_{x_0} - \Delta N_\psi(z_{x_0}), N_\psi(z_{x_0})).
$$
We have 
$$
\delta = \vol_{n+1}(\operatorname{epi}(\psi) \cap H_\delta^-) \geq  \vol_{n+1}(\mathcal{E} \cap H_\delta^-) .
$$
Now we continue as in (\ref{unten}) and after.   We arrive at 
\begin{eqnarray*}
\frac{\psi_\delta(x_0)  - \psi(x_0) }{\delta ^\frac{2}{n+2}}  &\leq& \frac{(1+ \beta_1\ \varepsilon)^\frac{2n}{n+2} }{(1- \beta_2\ \varepsilon)^{2\frac{n+1}{n+2}}}  \ c_{n+1} \  \left(\prod_{i=1}^n \frac{a_i}{\sqrt{a_{n+1}}} \right)^\frac{2}{n+1} \\
&\leq& \frac{(1+ \beta_1\ \varepsilon)^\frac{2n}{n+2} }{(1- \beta_2\ \varepsilon)^{2\frac{n+1}{n+2}}}  \ c_{n+1} \  \left(\prod_{i=k+1}^n \frac{a_i}{\sqrt{a_{n+1}}} \right)^\frac{2}{n+1} \  \varepsilon^\frac{2k}{n+2},
\end{eqnarray*}
where in the last inequality we have used 
 that for all $1 \leq i \leq k$, $a_i = \frac{1}{\varepsilon}$, 
\end{proof}
\vskip 4mm
\noindent
We require  a  uniform bound in $\delta$ for the quantity $\frac{\psi_\delta(x)- \psi(x)}{\delta^\frac{2}{n+2}}$ so that we can apply the Dominated Convergence theorem.
This is achieved in the next lemma.
\vskip 3mm 
 \begin{lemma}\label{lemma:interchange}
 Let  $\psi: \R^n \rightarrow \R $ be a  convex function such that
$0<\int_{\mathbb R^{n}}   e^{-\psi(x)}   dx  < \infty $.
  Then there exists $\delta_0$ such that for all $\delta < \delta_0$,  for  all $x \in \mathbb R^{n}$, 
\begin{eqnarray*}
0 \leq \frac{\psi_\delta(x)- \psi(x)}{\delta^\frac{2}{n+2}} \leq 2^\frac{3n+4}{n+2}
\left(\frac{ n+2}{\vol_{n} \left(B^n_2\right)}\right)^\frac{2}{n+2} 
\   \frac{(1+\|\nabla \psi (x)\|^2)^{\frac{1}{2}}} { r_\psi(x) ^\frac{n}{n+2}},
\end{eqnarray*}
where $r_\psi(x)$ is as in (\ref{r(x)}). 
\newline
Consequently we have for  all $\delta < \delta_0$, for all $x \in  \Omega_\psi$, 
\begin{eqnarray*}
0 \leq  \frac{ f(x) - f_\delta(x) } {\delta^{2/(n+2)}} 
\leq  2^\frac{3n+4}{n+2}
\left(\frac{ n+2}{\vol_{n} \left(B^n_2\right)}\right)^\frac{2}{n+2} 
\   \frac{(1+\|\nabla \psi (x)\|^2)^{\frac{1}{2}}} { r_\psi(x) ^\frac{n}{n+2}} \   f(x).
\end{eqnarray*}
\end{lemma}

\vskip 3mm
\begin{proof}  
Let $z_x = (x, \psi(x)) \in \partial\left(\operatorname{epi}(\psi)\right)$ and let $z_\delta =(x, \psi_\delta(x)$.  Let $r_\psi(x)= r_{\operatorname{epi}(\psi)}(z_x)$  be as in (\ref{r(x)}). If $N_{\psi} (z_x)$ is not unique, then $r_{\operatorname{epi}(\psi)}(z_x)=0$ and the inequality holds trivially.
 Moreover, if $\psi_\delta(x)= \psi(x)$, then $\frac{\psi_\delta(x)- \psi(x)}{\delta^\frac{2}{n+2}} =0$ and again, the inequality holds trivially.
\par
Thus we can assume that  $N_{\psi} (z_x)$ is  unique  and $\psi_\delta(x) > \psi(x)$. By Proposition \ref{properties}, there is a hyperplane $H_\delta$ such that $z_\delta \in H_\delta$ and
\begin{eqnarray}\label{interchange1}
\delta = \vol_{n+1}(H_\delta^- \cap \operatorname{epi}(\psi)) 
 \geq   \vol_{n+1}\left(H_\delta^- \cap B^{n+1}_2 \left( z_x-r_\psi(x) N_\psi(z_x), r_\psi(x) \right) \right).
\end{eqnarray}
We will estimate $ \vol_{n+1}\left(H_\delta^- \cap B^{n+1}_2 \left( z_x-r_\psi(x) N_\psi(z_x), r_\psi(x) \right) \right)$. We choose $\delta_0$ so small that  for all $\delta \leq \delta_0$,  $z_x \in H_\delta^-$.
\par
\noindent
We treat first the case 
\begin{equation}\label{assumption1}
\psi_\delta(x) - \psi(x)  \geq  r(x) \   \langle e_{n+1}, N_\psi(z_x) \rangle. 
\end{equation}
In this case  we have for all hyperplanes $H(z_\delta)$ through $z_\delta$  and such that $z_x \in H^-(z_\delta)$, 
\begin{eqnarray*} 
&& \hskip -20mm \vol_{n+1} \left(H^-(z_\delta) \cap B^{n+1}_2 \left(z_x-r_\psi(x) N_\psi(z_x) , r_\psi(x) \right) \right) \geq  \\
&&\hskip 20mm \vol_{n+1} \left(H_0^-(z_\delta) \cap B^{n+1}_2 \left(z_x-r_\psi(x) N_\psi(z_x) , r_\psi(x) \right) \right),
\end{eqnarray*} 
where $H_0(z_\delta)$ is this hyperplane orthogonal to $x$ and  such that both, $z_x$ and $z_\delta$ are  in $H_0(z_\delta)$.  We can estimate the latter from below by  the cone with base $\frac{1}{2} \left(\psi_\delta(x) - \psi(x)\right) \  B^n_2$ and height 
$h \geq \left(  \langle e_{n+1}, N_\psi(z_x) \rangle \right)^2 \  \frac{r_\psi(x)} {2}$. 
Hence, by (\ref{assumption1}), 
\begin{eqnarray*} 
&&\hskip -8mm  \vol_{n+1} \left(H_0^-(z_\delta) \cap B^{n+1}_2 \left(z_x-r_\psi(x) N_\psi(z_x) , r_\psi(x) \right) \right) \geq \\
&&\hskip 5mm \frac{\vol_{n} \left(B^n_2\right) } {2^{n+1} (n+1) } \left(  \langle e_{n+1}, N_\psi(z_x) \rangle \right) ^2 \ r_\psi(x) \  \left(\psi_\delta(x) - \psi(x) \right)^n \geq  \\
&&\hskip 8mm  \frac{\vol_{n} \left(B^n_2\right) } {2^{n+1}  (n+1) } \left(  \langle e_{n+1}, N_\psi(z_x) \rangle \right) ^2 \ r_\psi(x) \ \left(\psi_\delta(x) - \psi(x) \right)^\frac{n+2}{2} \left(\psi_\delta(x) - \psi(x) \right)^\frac{n-2}{2}  \geq \\
&&\hskip 10mm  \frac{\vol_{n} \left(B^n_2\right) } {2^{n+1}  (n+1) } \left(  \langle e_{n+1}, N_\psi(z_x) \rangle \right) ^\frac{n+2}{2}  \ r_\psi(x)^\frac{n}{2} \ \left(\psi_\delta(x) - \psi(x) \right)^\frac{n+2}{2}.
\end{eqnarray*}
Since  $ \langle e_{n+1}, N_{\psi}(z_{x_0})\rangle= \frac{1}{(1+\|\nabla \psi\|^2)^{\frac{1}{2}}}$,  we get with (\ref{interchange1}), 
\begin{eqnarray*} 
\frac{\psi_\delta(x)- \psi(x)}{\delta^\frac{2}{n+2}} \leq 
\left(\frac{2^{n+1} (n+1)}{\vol_{n} \left(B^n_2\right)}\right)^\frac{2}{n+2} 
\frac{(1+\|\nabla \psi (x)\|^2)^{\frac{1}{2}}} { r_\psi(x) ^\frac{n}{n+2}}.
\end{eqnarray*}
\par
\noindent
Now we treat the case
\begin{equation}\label{fall2}
0 < \psi_\delta(x)- \psi(x)  < r_\psi(x)\  \langle e_{n+1}, N_{\psi} (z_x) \rangle.
\end{equation}
For all hyperplanes $H(z_\delta)$ through $z_\delta$ such that $z_x \in H^-(z_\delta)$,  
$$ \vol_{n+1} \left(H^-(z_\delta) \cap B^{n+1}_2 \left(z_x-r_\psi(x) N_\psi(z_x) , r_\psi(x) \right) \right)$$ 
is minimal if the line segment $[z_\delta, z_x-r_\psi(x) N_\psi(z_x)]$ is orthogonal to the hyperplane $H(z_\delta)$. \\
Then 
$H^-(z_\delta) \cap B^{n+1}_2 \left(z_x-r(x) N_\psi(z_x) , r_\psi(x) \right)$ is a cap of $B^{n+1}_2 \left(z_x-r(x) N_\psi(z_x) , r_\psi(x) \right)$ of height $d$, where
$d = \text{dist} \left( z_\delta, \partial B^{n+1}_2 (z_x - r_\psi(x) N_{\psi} (z_x), r_\psi(x) \right)$. 
Let $h$ be the height of the cap $H_0^-(z_\delta) \cap B^{n+1}_2 \left(z_x-r(x) N_\psi(z_x) , r_\psi(x) \right)$ and let $\beta$ be the angle between the normal to $H_0$ and the line segment 
$[z_\delta, z_x-r_\psi(x) N_\psi(z_x)]$. 
\par
If $\beta =0$, then $d=h = \psi_\delta(x) -\psi(x)$ and 
$$
\delta \geq  \frac{ \vol_n(B^n_2) }{2^\frac{n}{2}(n+2)}  \   \left(  \psi_\delta(x) -\psi(x) \right) ^\frac{n+2}{2} \   r_\psi(x)^\frac{n}{2} 
$$
and thus 
\begin{eqnarray*} 
\frac{\psi_\delta(x)- \psi(x)}{\delta^\frac{2}{n+2}} \leq 2^\frac{n}{n+2}
\left(\frac{ n+2}{\vol_{n} \left(B^n_2\right)}\right)^\frac{2}{n+2} 
\    r_\psi(x) ^{-\frac{n}{n+2}}.
\end{eqnarray*}
Assume now \( \beta>0\). We first consider the case  $h < r_\psi(x)$. Then
$$
\cos\beta = \frac{r_\psi(x)-h}{r_\psi(x)-d}    \hskip 5mm \text{and} \hskip 5mm \sin \beta = \frac{r_\psi(x) \ \langle e_{n+1}, N_{\psi} (z_x) \rangle - (\psi_\delta(x) -\psi(x))}{r_\psi(x)-d}.
$$
From this we get that
\begin{eqnarray*}
d &=& r_\psi(x) - \left( (r_\psi(x)-h)^2 + \left(r_\psi(x) \ \langle e_{n+1}, N_{\psi} (z_x) \rangle - (\psi_\delta(x) -\psi(x))\right)^2\right)^\frac{1}{2}\\
&\geq &  r_\psi(x) \left( 1- \left(1 + \frac{(\psi_\delta(x) -\psi(x))^2}{r_\psi(x)^2}  - 2 \langle e_{n+1}, N_{\psi} (z_x) \rangle \frac{(\psi_\delta(x) -\psi(x))}{r_\psi(x)}  \right)^\frac{1}{2} \right)\\
&\geq &  \langle e_{n+1}, N_{\psi} (z_x) \rangle \left(\psi_\delta(x) -\psi(x)\right) \left(1 -  \frac{ \psi_\delta(x) -\psi(x) }{  2 \  r(x) \   \langle e_{n+1}, N_{\psi} (z_x) \rangle }\right) \\
&\geq&   \frac{1}{2} \langle e_{n+1}, N_{\psi} (z_x) \rangle \left(\psi_\delta(x) -\psi(x)\right).
\end{eqnarray*}
The latter inequality holds as  $\psi_\delta(x)- \psi(x)  < r_\psi(x)\  \langle e_{n+1}, N_{\psi} (z_x) \rangle$. 

Thus we get with Lemma \ref{lemma:cap},
\begin{eqnarray*} 
\delta \geq  \frac{ \vol_n(B^n_2) }{2^\frac{n}{2}(n+2)}  \   d^\frac{n+2}{2} \   r(x)^\frac{n}{2} 
\geq  \frac{ \vol_n(B^n_2) }{2^{n+1} (n+2) }   \langle e_{n+1}, N_{\psi} (z_x) \rangle^\frac{n+2}{2}\   \left(\psi_\delta(x) -\psi(x) \right) ^\frac{n+2}{2} \   r_\psi(x)^\frac{n}{2}, 
\end{eqnarray*}
which implies that 
\begin{eqnarray*} 
\frac{\psi_\delta(x)- \psi(x)}{\delta^\frac{2}{n+2}} \leq 2^{2   \frac{n+1}{n+2}}
\left(\frac{ n+2}{\vol_{n} \left(B^n_2\right)}\right)^\frac{2}{n+2} 
\   \frac{(1+\|\nabla \psi (x)\|^2)^{\frac{1}{2}}} { r_\psi(x) ^\frac{n}{n+2}}.
\end{eqnarray*}
If $h > r_\psi(x)$, then $\sin \beta$ is as above and $\cos\beta = \frac{h-r_\psi(x)}{r_\psi(x)-d}$. 
We continue as above.

\end{proof}

\noindent
The following lemmas were proved in \cite{SchuettWerner1990}.  
\vskip 2mm 

\begin{lemma} \cite{SchuettWerner1990} \label{SW-lemma}
 Let  $K$ be a  convex body in $\mathbb{R}^n$. Then we have for all $0 \leq \alpha < 1$, 
\begin{equation*}
\int _{\partial K} \frac{d \mu_K(x)} { r_K(x) ^\alpha}   < \infty,
\end{equation*}
where $r_K$ is as in (\ref{rC}).
\end{lemma}
\vskip 3mm 
\noindent

\noindent
\vskip 3mm 

\begin{lemma}\cite{SchuettWerner1990}
Let $K$ be a convex body in $\mathbb R^{n}$ that contains $B_{2}^{n}$.
Then we have for all $t$ with $0<t\leq1$ that
$\{x\in\partial K|r_{K}(x)\geq t\}$ is a closed set and 
$$
(1-t)^{n-1}\operatorname{vol}_{n-1}(\partial K)
\leq {\mathcal H}_{n-1}(\{x\in\partial K: r_{K}(x)\geq t\}),
$$
where ${\mathcal H}_{n-1}$ is the $(n-1)$-dimensional Hausdorff measure.
\end{lemma}
\noindent
When $K$ contains a Euclidean ball with radius $\lambda$ we get for all $0 < s \leq \lambda$, 
\begin{equation}\label{RollFunc1-1}
\left(1-\frac{s}{\lambda}\right)^{n-1}\operatorname{vol}_{n-1}(\partial K)
\leq {\mathcal H}_{n-1}(\{x\in\partial K: r_{K}(x)\geq s\}).
\end{equation}
It follows for all $\alpha$ with $0<\alpha<1$,
\begin{eqnarray}\label{RollFEst1}
&&\int_{\partial K}r_{K}^{-\alpha}d\mu_{\partial K} \nonumber  
=\int_{0}^{\infty}{\mathcal H}_{n-1}(\{x \in \partial K|r_{K}^{-\alpha}(x) > s\})ds
\nonumber\\
&&=\int_{0}^{1}{\mathcal H}_{n-1}(\{x \in \partial K|r_{K}^{-\alpha}(x) > s\})ds
+\int_{1}^{\infty}{\mathcal H}_{n-1}(\{x \in \partial K|r_{K}^{-\alpha}(x) > s\})ds
\nonumber\\
&&\leq\operatorname{vol}_{n-1}(\partial K)
+\operatorname{vol}_{n-1}(\partial K)\int_{1}^{\infty}(n-1)
s^{-\frac{1}{\alpha}}\lambda^{-1}
ds \nonumber \\
&&=\operatorname{vol}_{n-1}(\partial K)\left(1+\frac{n-1}{\lambda}
\left(\frac{1}{\alpha}-1\right)\right).
\end{eqnarray}

\noindent
The next lemma is the analogue of Lemma \ref{SW-lemma} in the present context. The definition of $r_\psi(x)$ is given in (\ref{r(x)}).

\begin{lemma}\label{integral}
 Let  $\psi: \R^n \rightarrow \R $ be a  convex function such that  
$ e^{-\psi}$ is integrable. Then we have for all $0 \leq \alpha < 1$, 
\begin{equation}\label{integral-1}
\int _{\mathbb R^{n}} \frac{(1+\|\nabla \psi (x)\|^2)^{\frac{1}{2}}} { r_\psi(x) ^\alpha} \  e^{-\psi(x)} \   dx  < \infty.
\end{equation}
In particular, this holds for $\alpha = \frac{n}{n+2}$. 
\end{lemma}

\vskip 3mm
\begin{proof} 
By Lemma \ref{integrate+limit}, there are $\gamma$ and $\beta$
such that for all $x\in\mathbb R^{n}$
\begin{equation}\label{integral-99}
\psi(x)\geq\gamma\|x\|_{}+\beta.
\end{equation}
We can assume that $\alpha >0$. The case $\alpha =0$ follows by Lemma \ref{lemma:psi->grad}. 
\par
We first assume that $\psi(x)<\infty$ for all $x\in\mathbb R^{n}$.
For $k \in \mathbb{N}$, let 
$$
E_k(\psi) =\{x \in \Omega_\psi: \psi(x) \leq 2^{k} \}.
$$ 
As $\psi$ is convex on $\mathbb R^{n}$,
 it is continuous on $\mathbb R^{n}$. Therefore the sets $E_k(\psi)$ are convex and closed.  
By Lemma \ref{integrate+limit} the sets $E_k(\psi)$  are
bounded.  For $k \in \mathbb{N}$ we put
$$
K_k = \{(x, y)  \in \mathbb{R}^{n} \times  \mathbb{R}:   \psi(x) \leq   y \leq 2^{k}\} 
$$
and 
\begin{equation}\label{integral-3}
A_{k}
=\operatorname{epi}(\psi)
\cap\{(x,t)\in\mathbb R^{n+1}:2^{k}\leq t\leq 2^{k+1}\}.
\end{equation}
Let $k_{0} \in \mathbb{N}$ be the smallest number such that
$K_{k_{0}}$ contains an $(n+1)$-dimensional Euclidean ball and such that
\begin{equation}\label{integral-112}
|\beta|\leq2^{k_{0}-1}.
\end{equation} 
We denote its radius by $\lambda$.
We may assume that $\lambda\leq 1$.
Thus, by convexity, for $k\geq k_{0}$ the sets
$A_{k}$ 
contain an $(n+1)$-dimensional Euclidean ball with
radius $\lambda$. Then
\begin{eqnarray*}
&&\int _{\mathbb R^{n}} \frac{(1+\|\nabla \psi (x)\|^2)^{\frac{1}{2}}} { r_\psi(x) ^\alpha} \  e^{-\psi(x)} \   dx    \\
&& =\int _{E_{k_0} (\psi)} \frac{(1+\|\nabla \psi (x)\|^2)^{\frac{1}{2}}} { r_\psi(x) ^\alpha} \  e^{-\psi(x)} \   dx   + \int _{(E_{k_0}(\psi))^C} \frac{(1+\|\nabla \psi (x)\|^2)^{\frac{1}{2}}} { r_\psi(x) ^\alpha} \  e^{-\psi(x)} \   dx.
\end{eqnarray*}
By (\ref{RollFEst1}), (\ref{integral-99}) and as
$r_{K_{k_{0}}}(x,\psi(x))\leq r_{\psi}(x)$ for all $x\in E_{k_{0}}$
\begin{eqnarray*}
&& \int _{E_{k_0} (\psi)} \frac{(1+\|\nabla \psi (x)\|^2)^{\frac{1}{2}}} { r_\psi(x) ^\alpha} \  e^{-\psi(x)} \   dx  \\
&&\leq  e^{-\beta} \int _{\partial K_{k_0}}  \frac{ d \mu _{K_{k_0}} (z) }{r_{K_{k_0}}(z)^\alpha}  
\leq e^{-\beta}\operatorname{vol}_{n}(\partial K_{k_{0}})\left(1+\frac{n}{\lambda}\left(\frac{1}{\alpha}-1\right)\right)    \\
&&\leq2 e^{-\beta}\left(1+\frac{n}{\lambda}\left(\frac{1}{\alpha}-1\right)\right)
\operatorname{vol}_{n}(\{(x,\gamma\|x\|+\beta):\ \gamma\|x\|+\beta\leq 2^{k_{0}}\}).
\end{eqnarray*}
The convex set $\{(x,\gamma\|x\|+\beta):\ \gamma\|x\|+\beta\leq 2^{k_{0}}\}$ is contained in the convex cylinder of height $2^{k_{0}}$ and radius
$\frac{2^{k_{0}}-\beta}{\gamma}$. By (\ref{integral-112}),
$\frac{2^{k_{0}}-\beta}{\gamma}\geq\frac{2^{k_{0}}}{\gamma}$.
Therefore, the surface area of the first set
is smaller than the surface area of this cylinder. Thus
\begin{eqnarray}\label{integral-98}
&& \int _{E_{k_0} (\psi)} \frac{(1+\|\nabla \psi (x)\|^2)^{\frac{1}{2}}} { r_\psi(x) ^\alpha} \  e^{-\psi(x)} \   dx \leq \\
&&4 e^{-\beta}\left(1+\frac{n}{\lambda}\left(\frac{1}{\alpha}-1\right)\right)
\left(\left(\frac{2^{k_{0}}-\beta}{\gamma}\right)^{n}\operatorname{vol}_{n}(B_{2}^{n})
+2^{k_{0}}\left(\frac{2^{k_{0}}-\beta}{\gamma}\right)^{n-1}\operatorname{vol}_{n-1}(\partial B_{2}^{n})\right),
\nonumber
\end{eqnarray}
which is finite. 
Let $G(\psi)=\{ (x, \psi(x)): x \in \mathbb{R}^n\}$ be the graph of $\psi$ and 
let 
\begin{equation}\label{integral-2}
\Gamma_k = G(\psi) \cap  \{y \in  \mathbb{R}^{n+1}: 2^{k}\leq  y_{n+1}\leq 2^{k+1} \} .
\end{equation}
We denote by $P$ the orthogonal projection onto $\mathbb R^{n}$.
Then  
\begin{eqnarray}\label{RollFEst2}
&&\int _{(E_{k_0}(\psi))^C} \frac{(1+\|\nabla \psi (x)\|^2)^{\frac{1}{2}}} { r_\psi(x) ^\alpha} \  e^{-\psi(x)} \   dx  
\leq \sum_{k=k_0}^ \infty  e^{-2^{k}}  \int_{P( \Gamma_k)}  \frac{ (1+\|\nabla \psi (x)\|^2)^{\frac{1}{2}}}{r_\psi(x)^\alpha} dx \nonumber \\
&& \leq \sum_{k=k_0}^ \infty  e^{-2^{k}} \int_{ \Gamma_k}   \frac{d \mu_{\Gamma_k}(z)}{r_{\Gamma_{k}}(z)^\alpha} 
\leq \sum_{k=k_0}^ \infty  e^{-2^{k}} \int_{\partial A_k}   \frac{d \mu_{A_k}(z)}{r_{A_{k}}(z)^\alpha} 
\nonumber  \\
&&\leq  \sum_{k=k_0}^ \infty  e^{-2^{k}} \operatorname{vol}_{n}(\partial A_{k})\left(1+
\frac{n}{\lambda}\left(\frac{1}{\alpha}-1\right)\right) .
\end{eqnarray}
The last inequality follows by (\ref{RollFEst1}), as $A_{k}$ contains a ball of radius $\lambda$.
\par
Recall that $H(x,\xi)$ denotes the hyperplane through $x$ and orthogonal
to $\xi$. Then
\begin{equation}\label{integral-4}
\partial A_{k}
=\Gamma_{k}\cup
(\operatorname{epi}(\psi)\cap H(2^{k}e_{n+1},e_{n+1}))
\cup(\operatorname{epi}(\psi)\cap H(2^{k+1}e_{n+1},e_{n+1})).
\end{equation}
We shall show that there is a constant $\alpha_{n}$ such that 
for all $k \geq k_0$,
\begin{equation}\label{integral-5}
\operatorname{vol}_{n}(\partial A_{k})
\leq \alpha_{n}\operatorname{vol}_{n}(\Gamma_{k}).
\end{equation}
As above
\begin{eqnarray}\label{integral-6}
\operatorname{vol}_{n}
(\operatorname{epi}(\psi)\cap H(2^{k}e_{n+1},e_{n+1}))
&\leq&\left(\frac{2^{k}-\beta}{\gamma}\right)^{n}\operatorname{vol}_{n}(B_{2}^{n})  \\
\operatorname{vol}_{n}
(\operatorname{epi}(\psi)\cap H(2^{k+1}e_{n+1},e_{n+1}))
&\leq&\left(\frac{2^{k+1}-\beta}{\gamma}\right)^{n}\operatorname{vol}_{n}(B_{2}^{n}).
\nonumber
\end{eqnarray}
To show (\ref{integral-5}), it is enough to show
\begin{eqnarray}\label{integral-7}
\operatorname{vol}_{n}(
\operatorname{epi}(\psi)\cap H(2^{k}e_{n+1},e_{n+1}))
&\leq& \alpha_{n}\operatorname{vol}_{n}(\Gamma_{k})  \\
\operatorname{vol}_{n}(
\operatorname{epi}(\psi)\cap H(2^{k+1}e_{n+1},e_{n+1}))
&\leq& \alpha_{n}\operatorname{vol}_{n}(\Gamma_{k}).
\nonumber
\end{eqnarray}
To do so, we apply the Schwarz symmetrization (see e.g. \cite{GardnerBook,SchneiderBook}) with axis $e_{n+1}$
to $A_{k}$. Then there is a rotationally invariant function 
$\tilde\psi:\mathbb R^{n}\to\mathbb R$ such that
\begin{equation}\label{integral-8}
\operatorname{Schw}(A_{k})
=\operatorname{epi}(\tilde \psi)
\cap\{(x,t)\in\mathbb R^{n+1}:2^{k}\leq t\leq 2^{k+1}\}.
\end{equation}
Let
$$
\tilde \Gamma_{k}
=G(\tilde \psi)\cap \{(x,t)\in\mathbb R^{n+1}:2^{k}\leq t\leq 2^{k+1}\}.
$$
Observe that
\begin{eqnarray}\label{integral-10}
&&\partial\operatorname{Schw}(A_{k})   \\
&&=\tilde \Gamma_{k}
\cup
(\operatorname{epi}(\tilde\psi)\cap H(2^{k}e_{n+1},e_{n+1}))
\cup(\operatorname{epi}(\tilde\psi)\cap H(2^{k+1}e_{n+1},e_{n+1})).
\nonumber
\end{eqnarray}
There are radii $r_{k}$ and $R_{k}$ with
\begin{eqnarray*}
\operatorname{epi}(\tilde\psi)\cap H(2^{k}e_{n+1},e_{n+1})
&=&B_{2}^{n}(2^{k}e_{n+1},r_{k})\cap H(2^{k}e_{n+1},e_{n+1})
\\
\operatorname{epi}(\tilde\psi)\cap H(2^{k+1}e_{n+1},e_{n+1})
&=&B_{2}^{n}(2^{k+1}e_{n+1},R_{k})\cap H(2^{k+1}e_{n+1},e_{n+1}).
\end{eqnarray*}
We have
\begin{eqnarray}\label{integral-14}
&&\operatorname{vol}_{n}(\operatorname{epi}(\psi)\cap H(2^{k}e_{n+1},e_{n+1}))   \\
&&=\operatorname{vol}_{n}(
\operatorname{epi}(\tilde\psi)\cap H(2^{k}e_{n+1},e_{n+1}))
=r_{k}^{n}\operatorname{vol}_{n}(B_{2}^{n})   
\nonumber
\end{eqnarray}
and
\begin{eqnarray}\label{integral-15}
&&\operatorname{vol}_{n}(\operatorname{epi}(\psi)\cap H(2^{k+1}e_{n+1},e_{n+1}))   \\
&&=\operatorname{vol}_{n}(
\operatorname{epi}(\tilde\psi)\cap H(2^{k+1}e_{n+1},e_{n+1}))
=R_{k}^{n}\operatorname{vol}_{n}(B_{2}^{n}) .
\nonumber
\end{eqnarray}
From the above considerations it follows that
\begin{equation}\label{integral-16}
\operatorname{vol}_{n}(\tilde \Gamma_{k})
\leq\operatorname{vol}_{n}( \Gamma_{k}).
\end{equation}
Indeed, since a Schwarz symmetrization reduces the surface area 
of a convex body
$$
\operatorname{vol}_{n}(\partial\operatorname{Schw}(A_{k}))
\leq\operatorname{vol}_{n}(\partial(A_{k}))
$$
and thus by (\ref{integral-4}) and (\ref{integral-10}) and as the unions are disjoint
\begin{eqnarray*}
&&\operatorname{vol}_{n}(\tilde\Gamma_{k})+
\operatorname{vol}_{n}(\operatorname{epi}(\tilde\psi)\cap H(2^{k}e_{n+1},e_{n+1})))+
\operatorname{vol}_{n}(\operatorname{epi}(\tilde \psi)\cap H(2^{k+1}e_{n+1},e_{n+1})))  \\
&&\leq\operatorname{vol}_{n}(\Gamma_{k})+
\operatorname{vol}_{n}(\operatorname{epi}(\psi)\cap H(2^{k}e_{n+1},e_{n+1})))+
\operatorname{vol}_{n}(\operatorname{epi}(\psi)\cap H(2^{k+1}e_{n+1},e_{n+1}))) .
\end{eqnarray*}
By (\ref{integral-14}) and (\ref{integral-15}), the inequality 
(\ref{integral-16}) follows.
\par
We show now that for some constant $b_{n}$
\begin{equation}\label{integral-66}
b_{n}\frac{R_{k}^{n}}{2^{n}}\operatorname{vol}_{n}( B_{2}^{n})
\leq
\operatorname{vol}_{n}(\tilde\Gamma_{k}).
\end{equation}
For this we show
\begin{equation}\label{integral-17}
(R_{k}^{n}-r_{k}^{n})\operatorname{vol}_{n}( B_{2}^{n})
\leq\operatorname{vol}_{n}(\tilde\Gamma_{k})
\end{equation}
and
\begin{equation}\label{integral-18}
r_{k}^{n-1}\operatorname{vol}_{n-1}(\partial B_{2}^{n}) 2^{k}
\leq\operatorname{vol}_{n}(\tilde\Gamma_{k}).
\end{equation}
To prove (\ref{integral-17}) we observe
$$
\operatorname{vol}_{n}(\tilde\Gamma_{k})
\geq\operatorname{vol}_{n}(\{x:r_{k}\leq x\leq R_{k}\}).
$$
To prove (\ref{integral-18}) we observe that the surface area
of the cylinder (without bottom and top)
$$
\{(x,t)\in\mathbb R^{n+1}:2^{k}\leq t\leq 2^{k+1}, \|x\|_{2}=r_{k}\}
$$
is less than the surface area of $\tilde\Gamma_{k}$. In order to prove
(\ref{integral-66}) we consider two cases. The first case is
$r_{k}\leq\frac{R_{k}}{2}$. By (\ref{integral-17}),
$$
\left(1-\frac{1}{2^{n}}\right)R_{k}^{n}
\operatorname{vol}_{n}( B_{2}^{n})
\leq\operatorname{vol}_{n}(\tilde\Gamma_{k}).
$$
The second case is $r_{k}\geq\frac{R_{k}}{2}$. By (\ref{integral-18}),
$$
\operatorname{vol}_{n}(\tilde\Gamma_{k})
\geq
r_{k}^{n-1}\operatorname{vol}_{n-1}(\partial B_{2}^{n}) 2^{k}
\geq\frac{R_{k}^{n-1}}{2^{n-1}}\operatorname{vol}_{n-1}(\partial B_{2}^{n}) 2^{k}.
$$
By (\ref{integral-6}) and (\ref{integral-15})
$$
\left(\frac{2^{k+1}-\beta}{\gamma}\right)^{n}\operatorname{vol}_{n}(B_{2}^{n})
\geq
\operatorname{vol}_{n}
(\operatorname{epi}(\psi)\cap H(2^{k+1}e_{n+1},e_{n+1}))
=R_{k}^{n}\operatorname{vol}_{n}(B_{2}^{n}).
$$
Therefore
$$
\frac{2^{k+1}-\beta}{\gamma}\geq R_{k}
$$
and consequently, by (\ref{integral-112})
\begin{eqnarray*}
\operatorname{vol}_{n}(\tilde\Gamma_{k})
&\geq&\frac{R_{k}^{n-1}}{2^{n-1}}
\operatorname{vol}_{n-1}(\partial B_{2}^{n}) 2^{k}
\geq \frac{R_{k}^{n}}{2^{n}}
\operatorname{vol}_{n-1}(\partial B_{2}^{n})\frac{\gamma}{1-\frac{\beta}{2^{k+1}}}
\\
&\geq& \frac{R_{k}^{n}}{2^{n}}
\operatorname{vol}_{n-1}(\partial B_{2}^{n})\frac{\gamma}{1+\frac{|\beta|}{2^{k+1}}}
\geq\frac{4\gamma}{5}\frac{R_{k}^{n}}{2^{n}}
\operatorname{vol}_{n-1}(\partial B_{2}^{n}).
\end{eqnarray*}
Thus we have shown (\ref{integral-66}). By (\ref{integral-4})
\begin{eqnarray*}
&&\operatorname{vol}_{n}(\partial(A_{k}))  \\
&&=
\operatorname{vol}_{n}(\Gamma_{k})+
\operatorname{vol}_{n}(
(\operatorname{epi}(\psi)\cap H(2^{k}e_{n+1},e_{n+1}))
+\operatorname{vol}_{n}(\operatorname{epi}(\psi)\cap H(2^{k+1}e_{n+1},e_{n+1})).
\end{eqnarray*}
By (\ref{integral-14}), (\ref{integral-15}), (\ref{integral-16}) and (\ref{integral-66})
$$
\operatorname{vol}_{n}(\partial(A_{k}))
\leq\operatorname{vol}_{n}(\Gamma_{k})
+r_{k}^{n}\operatorname{vol}_{n}(B_{2}^{n})+R_{k}^{n}\operatorname{vol}_{n}(B_{2}^{n})
\leq\operatorname{vol}_{n}(\Gamma_{k})
+\frac{2^{n+1}}{b_{n}}\operatorname{vol}_{n}(\Gamma_{k}),
$$
which shows (\ref{integral-5}).
By (\ref{RollFEst2}) and (\ref{integral-5})
\begin{eqnarray*}
&&\int _{(E_{k_0}(\psi))^c} \frac{(1+\|\nabla \psi (x)\|^2)^{\frac{1}{2}}} { r_\psi(x) ^\alpha} \  e^{-\psi(x)} \   dx  
\leq  \alpha_{n}\sum_{k=k_0}^ \infty  e^{-2^{k}} 
\operatorname{vol}_{n}(\Gamma_{k})\left(1+\frac{n}{\lambda}\left(\frac{1}{\alpha}-1\right)\right) \\
&&=  \alpha_{n}\left(1+\frac{n}{\lambda}\left(\frac{1}{\alpha}-1\right)\right)
\sum_{k=k_0}^ \infty  e^{-2^{k}} \int_{E_{k+1}\setminus E_{k}}
(1+\|\nabla \psi (x)\|^2)^{\frac{1}{2}}dx   .
\end{eqnarray*}
Since $2^{k}\leq\psi(x)\leq2^{k+1}$ on $E_{k+1}\setminus E_{k}$, we have
$\frac{\psi(x)}{2}\leq 2^{k}$. Therefore
\begin{eqnarray*}
&&\int _{(E_{k_0}(\psi))^c} \frac{(1+\|\nabla \psi (x)\|^2)^{\frac{1}{2}}} { r_\psi(x) ^\alpha} \  e^{-\psi(x)} \   dx  \\
&&\leq  \alpha_{n}\left(1+\frac{n}{\lambda}\left(\frac{1}{\alpha}-1\right)\right)
\sum_{k=k_0}^ \infty   \int_{E_{k+1}\setminus E_{k}}
(1+\|\nabla \psi (x)\|^2)^{\frac{1}{2}}e^{-\frac{\psi(x)}{2}}dx  \\
&&=  \alpha_{n}\left(1+\frac{n}{\lambda}\left(\frac{1}{\alpha}-1\right)\right)
   \int_{(E_{k_0}(\psi))^c}
(1+\|\nabla \psi (x)\|^2)^{\frac{1}{2}}e^{-\frac{\psi(x)}{2}}dx  \\
&&\leq  \alpha_{n}\left(1+\frac{n}{\lambda}\left(\frac{1}{\alpha}-1\right)\right)
   \int_{\mathbb R^{n}}
(1+\|\nabla \psi (x)\|^2)^{\frac{1}{2}}e^{-\frac{\psi(x)}{2}}dx .
\end{eqnarray*}
As in the proof of Lemma \ref{lemma:psi->grad}, we get
\begin{eqnarray*}
&&\int _{(E_{k_0}(\psi))^c} \frac{(1+\|\nabla \psi (x)\|^2)^{\frac{1}{2}}} { r_\psi(x) ^\alpha} \  e^{-\psi(x)} \   dx  \\
&&\leq \alpha_{n}\left(1+\frac{n}{\lambda}\left(\frac{1}{\alpha}-1\right)\right)\left(
\int_{\mathbb R^{n}}
e^{-\frac{\psi(x)}{2}}dx 
+4e^{-\frac{\beta}{2}}\operatorname{vol}_{n-2}(\partial B_{2}^{n-1})
n\left(\frac{2}{\gamma}\right)^{n-1}\Gamma(n-1)\right)
\end{eqnarray*}
and thus
\begin{equation}\label{integral-117}
\int _{\mathbb R^{n}} \frac{(1+\|\nabla \psi (x)\|^2)^{\frac{1}{2}}} { r_\psi(x) ^\alpha} \  e^{-\psi(x)} \   dx
\leq c(\beta,\gamma,n,k_{0}).
\end{equation}
\end{proof}

\vskip 5mm
The proof of Theorem  \ref{theo:f-deltafloat} and Corollary \ref{corollary}  follows immediately from these lemmas.
\par
\begin{proof}[Proof of Theorem \ref{theo:f-deltafloat}]
By the assumptions of the theorem, Lemmas \ref{lemma:interchange}, \ref{integral} and Lebesgue's Dominated Convergence Theorem we get that 
\begin{eqnarray*}
\lim _{\delta \rightarrow 0} \frac{ \int (f (x) - f_\delta(x) )  dx} {\delta^{2/(n+2)}} =
 \int \lim _{\delta \rightarrow 0} \frac{(f(x) - f_\delta(x) ) \ dx} {\delta^{2/(n+2)}} 
\end{eqnarray*}
Lemma \ref{lemma:f-delta} finishes the proof.

\end{proof}
\vskip 3mm
\begin{proof}[Proof of Corollary  \ref{corollary}]
The proof is  done in the same way using Lemmas \ref{lemma:f-delta},  \ref{lemma:interchange},    \ref{integral} and Lebesgue's Dominated Convergence Theorem.
\end{proof}

\vskip 3mm
\noindent
{\bf Acknowledgement}
\par
\noindent
This material is based upon work supported by the National Science Foundation under Grant No. DMS-1440140 while the  authors were in residence at the Mathematical Sciences Research Institute in Berkeley, California, during the Fall 2017 semester.  \\
We  want to thank the referee
for the careful reading and suggestions for improvement.

{}

\vskip 4mm
\noindent
Ben Li\\
{\small Department of Mathematics}\\
{\small Case Western Reserve University}\\
{\small Cleveland, Ohio 44106, U. S. A. }\\
{\small \tt bxl292@case.edu}\\ \\
\vskip 3mm
\noindent
Carsten Sch\"utt\\
{\small Mathematisches Seminar }\\
{\small Christian-Albrechts Universit\"at}\\
{\small 24098 Kiel, Germany }\\
{\small \tt schuett@math.uni-kiel.de}\\ \\
\vskip 3mm
\noindent
Elisabeth M. Werner\\
{\small Department of Mathematics \ \ \ \ \ \ \ \ \ \ \ \ \ \ \ \ \ \ \ Universit\'{e} de Lille 1}\\
{\small Case Western Reserve University \ \ \ \ \ \ \ \ \ \ \ \ \ UFR de Math\'{e}matique }\\
{\small Cleveland, Ohio 44106, U. S. A. \ \ \ \ \ \ \ \ \ \ \ \ \ \ \ 59655 Villeneuve d'Ascq, France}\\
{\small \tt elisabeth.werner@case.edu}\\ \\

\end{document}